\documentclass[12pt]{amsart}
\usepackage{amssymb,latexsym,amsmath,amsfonts}
\usepackage[mathscr]{eucal}
\usepackage{mathrsfs}
\usepackage{multicol}
\usepackage{graphicx}
\usepackage{amscd}
\usepackage{amsthm}
\usepackage{amsxtra}
\usepackage[nointegrals]{wasysym}
\usepackage{bm}
\usepackage{calc}
\usepackage{mathtools}
\usepackage{hyperref}
\usepackage{xcolor}
\hypersetup{colorlinks=true, citecolor=magenta}

\newtheorem{theorem}{Theorem}[section]
\newtheorem{lemma}[theorem]{Lemma}
\newtheorem{proposition}[theorem]{Proposition}

\theoremstyle{definition}

\newcommand{\beas}{\begin{eqnarray*}}
\newcommand{\eeas}{\end{eqnarray*}}
\newcommand{\bes} {\begin{equation*}}
\newcommand{\ees} {\end{equation*}}
\newcommand{\be} {\begin{equation}}
\newcommand{\ee} {\end{equation}}
\newcommand{\bea} {\begin{eqnarray}}
\newcommand{\eea} {\end{eqnarray}}

\newcommand{\eps}{\varepsilon}

\newcommand{\zbar}{\overline{z}}
\newcommand{\de} {\delta}
\newcommand{\om}{\omega}
\newcommand{\la}{\lambda}

\newcommand\partl[2]{\dfrac{\partial{#1}}{\partial{#2}}}

\newcommand{\Om}{\Omega}

\newcommand{\rea}{\operatorname{Re}}
\newcommand{\ima}{\operatorname{Im}}

\newcommand{\wt}{\widetilde}

\newcommand{\Cn}{\mathbb{C}^n}
\newcommand{\C} {\mathbb{C}} 

\newcommand{\rl}{\mathbb{R}}

\begin{document}
\title[Rational convexity and the Lagrangian condition for immersions]{A complex-analytic characterization of Lagrangian immersions in $\Cn$ with transverse double points}
\author{Purvi Gupta}
\address{Department of Mathematics, Indian Institute of Science, Bangalore 560012, India}
\email{purvigupta@iisc.ac.in}
\author{Rudranil Sahu}
\address{Department of Mathematics, Indian Institute of Science, Bangalore 560012, India}
\email{sahurudranil2016@gmail.com}
\date{\today}
\begin{abstract} Given a compact smooth totally real immersed $n$-submanifold $M\subset\Cn$ with only finitely many transverse double points, it is known that if $M$ is Lagrangian with respect to some K{\"a}hler form on $\Cn$, then it is rationally convex in $\Cn$ (Gayet, 2000), but the converse is not true (Mitrea, 2020). We show that $M$ is Lagrangian with respect to some K{\"a}hler form on $\Cn$ if and only if $M$ is rationally convex {\em and} at each double point, the pair of transverse tangent planes to $M$ satisfies the following diagonalizability condition: there is a complex linear transformation on $\Cn$ that maps the pair to $\left(\rl^n,(D+i)\rl^n\right)$ for some real diagonal $n\times n$ matrix $D$.
\end{abstract}
\maketitle

\section{Introduction}
A compact set $K\subset \Cn$ is said to be rationally convex if its complement in $\Cn$ is a union of complex hypersurfaces in $\Cn$. Rational convexity endows a compact set $K\subset\Cn$ with the approximation-theoretic feature that any function holomorphic on some neighborhood of $K$ is uniformly approximable on $K$ by restrictions of rational functions on $\Cn$ with poles off of $K$. A remarkable symplectic characterization of rational convexity was established by Duval ($n=2$) and Duval--Sibony ($n\geq 2$) for compact totally real submanifolds of $\Cn$, see \cite{Du91,Du94,DuSi95}. They proved that the image of a totally real smooth embedding $\iota:M\rightarrow\Cn$ of a compact smooth manifold $M$ is rationally convex in $\Cn$ if and only if the image is isotropic, or Lagrangian, if $\text{dim}_{\rl}M=n$, with respect to some K{\"a}hler form on $\Cn$, i.e., there is a strictly plurisubharmonic function $\rho:\Cn\rightarrow\rl$ such that $\iota^*dd^c\rho=0$. In \cite{Ga00}, Gayet considered the case where $M$ is $n$-dimensional and $\iota:M\rightarrow\Cn$ is a totally real immersion that self intersects in only finitely many transverse double points. He showed that if $\iota(M)$ is Lagrangian with respect to some K{\"a}hler form on $\Cn$, then it is rationally convex. This result was extended to immersions with finitely many quadratic self-tangencies by Duval--Gayet in \cite{DuGa08}. In \cite{Mi20}, Mitrea gave examples to show that the converse may not hold in the presence of self-intersections. In particular, there exist pairs $(R,S)$ of $n$-dimensional totally real planes in $\Cn$ intersecting only at the origin such that $R\cup S$ is rationally convex, but $R\cup S$ is not Lagrangian with respect to {\em any} K{\"a}hler form on $\Cn$. Nevertheless, Mitrea showed that if $\iota:M\rightarrow\Cn$ is a totally real immersion with finitely many finite self-intersection points and $\iota(M)$ is rationally convex, then there is a nonnegative $(1,1)$-form $\omega$, positive outside an arbitrarily small neighborhood of the self-intersection set of $\iota(M)$, so that $\iota^*\omega=0$, see \cite[Theorem 1.1]{Mi20}.

In this article, we consider the case of immersed $n$-dimensional submanifolds of $\Cn$ that admit only isolated self-intersection points, each of which is a transverse double point. This was the case considered by Gayet in \cite{Ga00}. We give the precise additional hypothesis required at the double points to obtain a converse of Gayet's theorem. For this, we need a definition. Let $\iota:M\rightarrow\Cn$ be an immersion of a smooth real $n$-dimensional manifold $M$ into $\Cn$. Assume that $\iota$ is totally real, i.e., for every $z\in M$, the real subspace $\iota_*(T_zM)$ of $\Cn$ contains no complex lines. Let $p\in\iota(M)$ be an isolated transverse double point of $\iota(M)$, and $z,w\in \iota^{-1}(p)$ be distinct pre-images of $p$. Since $\iota_*(T_zM)$ is a totally real plane, there is a nonsingular complex linear transformation $G:\Cn\rightarrow\Cn$ that maps $\iota_*(T_zM)$ to $\rl^n=\{z\in\Cn:\ima z=0\}$. Due to transversality, there is a real $n\times n$ matrix $A$ (depending on $G$) such that $G$ maps $\iota_*\left(T_wM\right)$ to $(A+i)\rl^n=\{z\in\Cn:\rea z=Ax \text{ and }\ima z=x\text{ for some }x\in\rl^n\}$; see \cite[\S 2]{We88}. We say that $p$ is a {\em diagonalizable} double point of $\iota(M)$ if $A$ is diagonalizable over the reals. This condition is independent of the choice of $G$. Our main result is as follows. 

\begin{theorem}\label{T:main} Let $M$ be a compact smooth manifold of real dimension $n$. Let $\iota:M\rightarrow\Cn$ be a totally real immersion such that $\iota(M)$ is a smooth submanifold of $\Cn$, except at finitely many points $p_1,...,p_m\in\iota(M)$, which are transverse double points of $\iota(M)$. Then, $\iota(M)$ is Lagrangian with respect to some K{\"a}hler form on $\Cn$ if and only if $\iota(M)$ is rationally convex and $p_1,...,p_m$ are diagonalizable double points of $\iota(M)$.      
\end{theorem}

The forward implication follows from Gayet's aforementioned result in \cite{Ga00}, and an inspection of the linear model of a transverse double point, i.e., the union of two $n$-dimensional totally real planes $R$ and $S$ in $\Cn$ that intersect only at the origin. In \cite{We88}, Weinstock showed that $R\cup S$ is locally polynomially (and rationally) convex at the origin if and only if, after a complex linear transformation, $R=\rl^n$ and $S=(A+i)\rl^n$ for some real $n\times n$  matrix $A$ that has no purely imaginary eigenvalues of modulus greater than one. In Section~\ref{S:planes}, we show that $R\cup S$ is Lagrangian with respect to some K{\"a}hler form on $\Cn$ if and only if $A$ is diagonalizable over the reals, see Proposition~\ref{P:planes}. For the converse, we construct a local K{\"a}hler form near each diagonalizable double point, with respect to which the immersed manifold is locally Lagrangian, see Section~\ref{S:local}, and patch these local K{\"a}hler forms with the global ``degenerate" K{\"a}hler form granted by Mitrea's result in \cite{Mi20} to complete the proof of Theorem~\ref{T:main}, see Section~\ref{S:proof}.

\section{Preliminaries}\label{S:prelim}
In this section, we collect the notation that will be used throughout this paper. 

A point $z\in\Cn$ is denoted in complex coordinates by $(z_1,...,z_n)$, where  $z_j=x_j+iy_j\in\C$, $1\leq j\leq n$. The origin in any Euclidean space is denoted by $0$. 
\smallskip 

\noindent {\bf Fr{\'e}chet derivatives.}  Vectors in $\rl^n$ are treated, by default, as column matrices. For a smooth $\rl^m$-valued map $f$ on an open set $\Om\subset\rl^n$, the Fr{\'e}chet derivative of $f$ at $z\in\Om$ is viewed as an $m\times n$ matrix, and is denoted by $(Df)(z)$. Identifying real $m\times n$ matrices with $\rl^{mn}$ in the natural way, the Fr{\'e}chet derivative of $Df:\Om\rightarrow\rl^{mn}$ at $z\in\Om$ is denoted by $(D^2f)(z)$, and is viewed as an $(mn)\times n$ matrix. In the special case where $m=1$, the $1\times n$ row matrix $(Df)(z)$ is denoted by $(\nabla f)(z)$ and the $n\times n$ matrix $(D^2f)(z)$ is denoted by $(\operatorname{Hess}f)(z).$ If $f$ is defined on an open set in $\Cn$, 
\beas
(\operatorname{Hess}_{xx}f)(x,y)&=&(\operatorname{Hess}f_y)(x),\quad \text{where }f_y:x\mapsto f(x,y),\\
(\operatorname{Hess}_{yy}f)(x,y)&=&(\operatorname{Hess}f_x)(y),\quad \text{where }f_x:y\mapsto f(x,y),\\
(\operatorname{Hess}_{xy}f)(x,y)&=&(DF_x)(y), \quad \text{where }F_x:y\mapsto (\nabla_xf)(x,y),\\
(\operatorname{Hess}_{yx}f)(x,y)&=&(DF_y)(x), \quad \text{where }F_y:x\mapsto (\nabla_yf)(x,y).
\eeas
Note that $(\operatorname{Hess}_{yx}f)=(\operatorname{Hess}_{xy}f)^T$.
\smallskip

\noindent {\bf Matrices and matrix-valued functions.} Let $M_{m\times n}(\rl)$ denote the set of all real $m\times n$ matrices. The $n\times n$ identity matrix is denoted by $I_n$. The zero matrix of any size is denoted by $0$. The transpose of a matrix $A$ is denoted by $A^T$. Matrix multiplication is denoted by $\cdot$, i.e., for compatible matrices $A$ and $B$, $A\cdot B$ denotes their matrix product. Given a smooth $(m\times n)$-matrix valued function $A$ and a smooth $(n\times p)$-matrix valued function $B$ on an open set $\Om\subset\rl^k$, we use the following product rule for differentiation:
\be\label{E:prod}
D(A\cdot B)=\left(B^T\otimes I_m\right)\cdot (DA)+\left(I_p\otimes A\right)\cdot (DB),
\ee
where $\otimes$ denotes the Kronecker product between matrices given by 
    \bes
        (a_{jk})_{m\times n}\otimes B_{r\times s}=\begin{pmatrix} a_{11}B & \cdots & a_{1n}B\\
        \vdots &\ddots &\vdots\\
            a_{m1}B & \cdots & a_{mn}B\end{pmatrix}_{mr\times ns}.
    \ees 

\noindent{\bf Vector fields.} Given a row matrix $V=(V_1,...,V_m)$ of real vector fields on an open set $\Om\subset\rl^n$ and an $(m\times p)$ matrix of functions $A=(a_{jk})$ on $\Om$, $V\cdot A$ denotes the row matrix of real vector fields given by 
    \be\label{E:VF}
        V\cdot A=\left(\sum_{k=1}^m{a_{1k}V_k},...,\sum_{k=1}^m{a_{pk}V_k}\right).
    \ee
If $G$ is a smooth map on $\Om$, then 
    \be\label{E:PushVF}
        G_*V=(G_*V_1,...,G_*V_n). 
    \ee

\noindent{\bf The $\boldsymbol{dd^c}$ operator.} The Dolbeault operator and its complex conjugate on $\Cn$ are given by
\bes
    \partial f=\sum_{j=1}^n\partl{}{z_j}dz_j\quad\text{and}\quad  \overline{\partial} f=\sum_{j=1}^n\partl{}{\zbar_j}d\zbar_j. 
\ees
Then, the usual exterior derivative (acting on $0$-forms) may be expressed as $d=\partial +\overline\partial $. Consider the differential operator $d^c=i(\overline\partial-\partial)$. Then, for a smooth $\rl$-valued function $f$ on some open set $\Om\subset\Cn$,
    \bes        d^cf=\sum_{j=1}^n\left(\partl{f}{x_j}dy_j-\partl{f}{y_j}dx_j\right),
    \ees
in terms of real coordinates. 
We use the notation $dd^cf>0$ on $\Om$ to denote the fact that the alternating $(1,1)$-form $(dd^cf)(z)$ is (strictly) positive definite for each $z\in\Om$. The following chain rule for the operator $dd^c$ will be useful. Let $G=(G_1,...,G_n)$ be a smooth $\Cn$-valued map on open set $\Om\subset\Cn$ such that $G(\Om)$ is an open set in $\Cn$. For a smooth function $f:G(\Om)\rightarrow\rl$, 
    \bea\label{E:chain}
        dd^c(f \circ G) &=&\ 2i \sum_{j,k=1}^n \left( \frac{\partial^2 f}{\partial z_j \partial \zbar_k} \circ G \right) 
\left(\partial G_j \wedge \overline{\partial}\, \overline{G}_k - \overline{\partial} G_j \wedge \partial \overline{G}_k \right)\notag\\ 
&&\quad + 2i \sum_{j,k=1}^n \left( \frac{\partial^2 f}{\partial z_j \partial z_k} \circ G \right) 
\partial G_j \wedge \overline\partial G_k
+ 2i \sum_{j,k=1}^n \left( \frac{\partial^2 f}{\partial \zbar_j \partial \zbar_k} \circ G \right) 
\partial\, \overline{G}_j \wedge \overline\partial\,\overline{G}_k \\
&&\quad\quad +2i \sum_{k=1}^n \left( \frac{\partial f}{\partial z_k} \circ G \right) \partial\overline{\partial} G_k 
+ 2i \sum_{k=1}^n \left( \frac{\partial f}{\partial \zbar_k} \circ G \right) \partial\overline{\partial} \,\overline{G}_k.\notag
    \eea

\section{The Lagrangian property for a union of planes}\label{S:planes}

In this section, we give a necessary and sufficient condition for a pair of real $n$-dimensional totally real planes intersecting only at the origin to have the property that their union is Lagrangian with respect to a K{\"a}hler form on $\Cn$. As observed in \cite[\S~2]{We88}, any nonsingular complex linear transformation that maps one of the two planes to $\rl^n=\{z\in\Cn:\ima z=0\}$, maps the other plane to $S(A)=(A+i)\rl^n$ for some $A\in M_{n\times n}(\rl)$. We note an elementary result that will be used throughout the paper. 

\begin{lemma}\label{L:bihol} Let $\iota:M\subset\Cn$ be a smooth embedding of a smooth real $n$-dimensional manifold $M$. Let $\omega$ be a K{\"a}hler form on some neighborhood of $\iota(M)$ such that $\iota(M)$ is Lagrangian with respect to $\omega$, i.e. $i^*\omega=0$. Let $G$ be a biholomorphism on a neighborhood of $\iota(M)$. Then, $(G\circ\iota)(M)$ is Lagrangian with respect to the K{\"a}hler form ${G^{-1}}^*\omega$.
\end{lemma}

By the above lemma, it suffices to assume that the given pair of totally real planes is of the form $(\rl^n,S(A))$ for some $A\in M_{n\times n}(\rl)$. We then have the following result. 

\begin{proposition}\label{P:planes} Let $A\in M_{n\times n}(\rl)$ and $S(A)=(A+i)\rl^n$. Then, $\rl^n\cup S(A)$ is Lagrangian with repsect to some K{\"a}hler form on $\Cn$ if and only if $A$ is diagonalizable over the reals. 
\end{proposition}

Before we proceed with the proof of this result, we observe that it suffices to work with a real Jordan form of $A$. This is owing to Lemma~\ref{L:bihol} and the following lemma. 

\begin{lemma}[{\cite[Lemma~2]{We88}}]\label{L:diagonalize}    Let $P\in M_{n\times n}(\rl)$ be nonsingular. Then, the complex linear transformation on $\Cn$ given by $P$ maps $\rl^n$ onto $\rl^n$ and $S(A)$ onto $S(PAP^{-1})$. 
\end{lemma}
\begin{proof} That $P$ maps $\rl^n$ bijectively onto $\rl^n$ is clear. Let $z\in S(PAP^{-1})$. Then, there is an $x=(x_1,...,x_n)\in \rl^n$ such that \bes
z=\left(PAP^{-1}\right)x+ix=\left(PAP^{-1}\right)x+i(PP^{-1})x=
P\left((A+i)(P^{-1}x)\right)\in P\left( S(A)\right)\ees
since $P^{-1}x\in \rl^n$. Thus, $S(PAP^{-1})\subset P(S(A))$. Applying this to $P^{-1}$ and $PAP^{-1}$ instead of $P$ and $A$, respectively, we obtain that $S(A)\subset P^{-1}S(PAP^{-1})$, which gives the other inclusion.
\end{proof}

\noindent {\em Proof of Proposition~\ref{P:planes}.} We first record a general formula that will be used throughout this proof. Let 
    \bes
        \iota:(t_1,...,t_n)\mapsto (t_1,...,t_n),\quad (t_1,...,t_n)\in\rl^n
    \ees
be a parametrization of $\rl^n$. For a fixed $A=(a_{jk})\in M_{n\times n}(\rl)$, let
    \bes
        \iota_A:(t_1,...,t_n)\mapsto\left(
        \sum_{k=1}^na_{1k}t_k+it_1,...,\sum_{k=1}^na_{nk}t_k+it_n\right),\quad (t_1,...,t_n)\in\rl^n,
    \ees
be a parametrization of $S(A)$. Given a K{\"a}hler form (with constant coefficients)
    \bes        \omega=\sum_{j,k=1}^nh_{jk}dz_j\wedge d\bar{z}_k
    \ees
on $\Cn$, where $h=(h_{jk})\in M_{n\times n}(\C)$ is a positive definite Hermitian matrix, we have that
    \bes
       \iota^*\omega= \sum_{1\leq j<k\leq n} (h_{jk}-h_{kj})dt_j\wedge dt_k= 2i\sum_{1\leq j<k\leq n} \left(\ima h_{jk}\right) dt_j\wedge dt_k,
    \ees
and
    \beas
    \iota_A^*\omega&=&\sum_{j,k=1}^n
        h_{jk}\left(dt_j\wedge dt_k+\sum_{t,m=1}^n(a_{jm}a_{k\ell})dt_m\wedge dt_\ell+i\sum_{\ell=1}^n a_{k\ell} dt_j\wedge dt_\ell
        -i\sum_{m=1}^na_{jm}dt_m\wedge dt_k\right)\\
        &=&\sum_{j,k=1}^n\left(h_{jk}+        \sum_{r,s=1}^nh_{rs}a_{rj}a_{sk}+i\sum_{r=1}^n h_{jr}a_{rk}-i\sum_{s=1}^nh_{sk}a_{sj}\right)dt_j\wedge dt_k\\
        &=&2i\sum_{1\leq j< k\leq n}\left(h_{jk}-h_{kj}+\sum_{r,s=1}^n\left(h_{rs}-h_{sr}\right)a_{rj}a_{sk}+\sum_{r=1}^n \left(h_{jr}+h_{rj}\right)a_{rk}-\sum_{s=1}^n(h_{sk}+h_{ks})a_{sj}\right)dt_j\wedge dt_k.   
    \eeas
Thus, $\rl^n\cup S(A)$ is Lagrangian with respect to $\omega$ if and only if 
    \bea
    \ima h_{jk}&=&0,\quad 1\leq j<k\leq n, \label{E:Lag1}\\
   c_{jk}(A)= \sum_{r=1}^n\left((\rea h_{jr})a_{rk}-(\rea h_{kr})a_{rj}\right)&=&0,\quad 1\leq j<k\leq n.\label{E:Lag2}
    \eea

Let us proceed to prove the claim. First, assume that $A$ is diagonalizable over the reals. By Lemmas~\ref{L:bihol} and \ref{L:diagonalize}, it suffices to consider a real Jordan form $J(A)$ of $A$, which by assumption is a diagonal matrix with real diagonal entries. From \eqref{E:Lag1} and \eqref{E:Lag2}, it follows that $\rl^n\cup S(J(A))$ is Lagrangian with respect to the standard K{\"a}hler form on $\Cn$. 

Conversely, suppose $A\in M_{n\times n}(\rl)$ is such that $\rl^n\cup S(A)$ is Lagrangian with respect to some K{\"a}hler form on $\Cn$. Then, by Lemmas~\ref{L:bihol} and ~\ref{L:diagonalize}, so is $\rl^n\cup S(J(A))$, where $J(A)$ is a real Jordan form of $A$. Suppose $A$ is not diagonalizable. Then, $J(A)$ can be chosen to be of the form 
\bes
J(A)= \begin{pmatrix}
    A_1 &  &       &\\
        &  &\ddots &\\
        &  &       & A_m\\
\end{pmatrix}
\ees
where $A_1$ is either of the form 
\be\label{E:realroot}
A_1=\begin{pmatrix}
    \la & 1 &  & \\
     & \la & \ddots  & \\
     & & \ddots & 1 \\
     &  &  & \la
\end{pmatrix}_{r\times r}
\ee
for some real eigenvalue $\la$ of $A$ and $r\geq 2$, or of the form
\be\label{E:comproot}
\begin{pmatrix}
    C & I_2 &  & \\
     & C & \ddots  & \\
     & & \ddots & I_2 \\
     &  &  & C
\end{pmatrix}_{2k\times2k}
\quad \text{with}\quad 
C=
\begin{pmatrix}
  s & -t\\
  t & s 
\end{pmatrix}
\ee
for some complex eigenvalue $s+it$ of $A$ and $k\geq 1$. We consider the two cases separately.

\noindent \textbf{Case 1.} Assume that $A_1$ is of the form \eqref{E:realroot}. Let
\bes        \omega=\sum_{j,k=1}^nh_{jk}dz_j\wedge d\bar{z}_k
\ees
be a K{\"a}hler form on $\Cn$ so that $\rl^n\cup S(J(A))$ is Lagrangian with respect to $\omega$. Since both $\rl^n$ and $S(J(A))$ are linear subspaces of $\rl^{2n}$, it suffices to consider the case where the $h_{jk}$'s are constant functions. From \eqref{E:Lag2} and the fact that $h_{12}=\overline {h_{21}}$, we obtain that
    \bes
       c_{12}(J(A))=h_{11}+\lambda\left(\rea h_{12}-\rea h_{21}\right)=h_{11}=0
    \ees
which contradicts the fact that 
$(h_{jk})$ is a positive definite matrix.
\medskip

\noindent \textbf{Case 2.} Assume that $A_1$ is of the form \eqref{E:comproot}. Let
\bes        \omega=\sum_{j,k=1}^nh_{jk}dz_j\wedge d\bar{z}_k
\ees
be a K{\"a}hler form on $\Cn$ so that $\rl^n\cup S(J(A))$ is Lagrangian with respect to $\omega$. Since both $\rl^n$ and $S(J(A))$ are linear subspaces of $\rl^{2n}$, as before, it suffices to consider the case where the $h_{jk}$'s are constant functions. Once again, we compute $c_{12}(J(A))$.  From \eqref{E:Lag2} and the fact that $h_{12}=\overline {h_{21}}$, we obtain that
    \bes
       c_{12}(J(A))=-h_{11}t+\rea h_{12}s-\rea h_{21}s-h_{22}t=-t(h_{11}+h_{22})=0,
    \ees
which contradicts the fact that 
$(h_{jk})$ is a positive definite matrix.

Thus, $A$ must be diagonalizable over the reals. 
\qed

\section{Construction of the K{\"a}hler form at diagonalizable transverse points}\label{S:local}

In this section, we inspect a maximal totally real immersed submanifold of $\Cn$ near a single diagonalizable transverse double point, and show that the submanifold is locally Lagrangian with respect to a K{\"a}hler form on a neighborhood of the double point. In order to patch such local K{\"a}hler forms together, as will be done later, we need something stronger: we produce a local K{\"a}hler form $dd^cf$ such that pull-back of $d^cf$ vanishes on the submanifold near the double point. 



\begin{lemma}\label{L:local} Let $\iota:M\rightarrow\Cn$ be a smooth totally real immersion of a smooth $n$-dimensional manifold $M$. Let $p\in \iota(M)$ be an isolated diagonalizable double point of $\iota(M)$. Then, there exist a neighborhood $U$ of $p$ in $\Cn$ and a smooth function $f:U\rightarrow\rl$ such that
\begin{itemize}
    \item [(i)] $\iota^*(d^cf)=0$ on $\iota^{-1}(U)$, and
    \item [(ii)] $dd^cf>0$ on $U$.
\end{itemize}
\end{lemma}

\begin{proof}
Let $\mathcal R$ and $\mathcal S$ denote the two local branches of $\iota(M)$ at the double point $p$. Due to the hypothesis on $p$, after a complex affine change in coordinates, we may assume that $p$ is the origin $0$, $T_0\mathcal R=\rl^n$, and $T_0\mathcal S=(A+i)\rl^n$, where 
\bes A=(a_{jk})_{n\times n}=\operatorname{diag}(\lambda_1,...,\lambda_n),\quad \text{for some } \lambda_1,...,\lambda_n\in\rl.
\ees
By the implicit function theorem, there exist open balls $\mathcal B_1,...,\mathcal B_4\subset\rl^n$ centered at the origin, and smooth functions $\phi=(\phi_1,...,\phi_n):\mathcal B_1\rightarrow \mathcal B_3$ and $\psi=(\psi_1,...,\psi_n):\mathcal B_2\rightarrow \mathcal B_4$, with 
\be\label{E:phipsi}
\phi(0)=\psi(0)=D\phi(0)=D\psi(0)=0,
\ee
such that $\mathcal R=\left\{z\in \mathcal B_1+i\mathcal B_3: z=t+i\phi(t)\right\}$ and $\mathcal S=\left\{z\in \mathcal B_2+i\mathcal B_4 : z=(A+i)t+\psi(t)\right\}$. Parametrize $\mathcal R$ and $\mathcal S$ locally via the maps
\beas
\imath:(t_1,...,t_n)&\mapsto& (t_1+i\phi_1(t),...,t_n+i\phi_n(t)),\quad t=(t_1,...,t_n)\in \mathcal B_1\\
\jmath:(t_1,...,t_n)&\mapsto& (\la_1t_1+\psi_1(t)+it_1,...,\la_nt_n+\psi_n(t)+it_n),\quad t=(t_1,...,t_n)\in \mathcal B_2.
\eeas

In the rest of the proof, we view $\Cn$ as $\rl^{2n}$ via the identification
\bes
    z=(z_1,...,z_n)\leftrightarrow (x,y),
\ees
where $x=(x_1,...,x_n)$ and $y=(y_1,...,y_n)$. The projection maps $(x,y)\mapsto x$ and $(x,y)\mapsto y$ are denoted by $\pi_x$ and $\pi_y$, respectively. Under this identification, we have that 
    \beas
    \imath(t)&=&(t,\phi(t))\\
    \jmath(t)&=&(A\cdot t+\psi(t),t).
    \eeas
The origin in $\Cn$ is henceforth denoted by $(0,0)$. 

We produce a smooth function $r$ on a neighborhood of the origin in $\Cn$ such that the function 
    \be\label{E:f}
        f(x,y)=\Vert x\Vert^2+\Vert y\Vert^2+r(x,y)
    \ee
satisfies
\begin{itemize}
    \item [(a)] $\imath^*(d^cf)=0$, i.e., 
            \bes
        \sum_{j=1}^n
       \left(\partl{f}{x_j}(t,\phi(t))\left(\sum_{k=1}^n\partl{\phi_j}{t_k}(t)dt_k\right)-\partl{f}{y_j}(t,\phi(t))dt_j\right)=0,
       \ees
    \item [(b)] $\jmath^*(d^cf)(t)=0$, i.e.,
        \bes
        \sum_{j=1}^n
       \left(\partl{f}{x_j}(A\cdot t+\psi(t),t)dt_j-\partl{f}{y_j}(A\cdot t+\psi(t),t)\left(\lambda_jdt_j+\sum_{k=1}^n\partl{\psi_j}{t_k}(t)dt_k\right)\right)=0,
       \ees
    \item [(c)] $dd^cf>0$,
\end{itemize}
on some neighborhood of the origin. We will express row matrices of vector fields on $\rl^{2n}$ in terms of the following two fundamental row matrices of vector fields:
\bes
\nabla_x=\left(\partl{}{x_1},...,\partl{}{x_n}\right)\quad \text{and}\quad \nabla_y=\left(\partl{}{y_1},...,\partl{}{y_n}\right).
\ees
In terms of $r$, (a) and (b) translate to
\bea
(Vr)(t,\phi(t))=p(t) \label{E:PDE1}\\
 (Wr)(A\cdot t+\psi(t),t)=q(t), \label{E:PDE2}
\eea
where $V$ and $W$ are each row matrices of vector fields given by 
    \beas
        V&=&\nabla_x\cdot (D\phi\circ\pi_x)-\nabla_y\\
        W&=&\nabla_x-\nabla_y\cdot (A+D\psi\circ\pi_y),
    \eeas
see the notation established in \eqref{E:VF}, and $p$ and $q$ are row matrices of functions given by
    \beas
        p(t)&=&2\phi(t)^T-2t^T\cdot (D\phi)(t)\\
        q(t)&=&2t^T\cdot (D\psi)(t)-2\psi(t)^T.
    \eeas
It is worth noting here that, owing to \eqref{E:phipsi},
\be\label{E:pq}
p(0)=q(0)=0\quad \text{and}\quad (Dp)(0)=(Dq)(0)=0.
\ee
It suffices to produce a smooth $r$ that satisfies \eqref{E:PDE1} and \eqref{E:PDE2} for $t$ on some open neighborhood of $0\in\rl^n$, and $(dd^cr)(0)=0$. Then, (a)-(c) follow for $f$ defined as in \eqref{E:f}.

For this, we introduce a change of variables $(x,y)\mapsto (u,v)$ that ``straightens" the submanifold $\mathcal S$ to the plane $\{u=0\}$ and the row matrix of vector fields  $W$ along $\mathcal S$ to the row matrix $\nabla_u$ along $\{u=0\}$. Let $\Theta(u,v)=(x(u,v),y(u,v))$ be given by
    \beas
    x(u,v)&=&u+A\cdot v+\psi(v)\\
    y(u,v)&=&v-\left(A+(D\psi)(v)\right)\cdot u.
    \eeas
Clearly, $\Theta$ is a smooth map (where defined) with $\Theta(0,0)=(0,0)$ and 
\bes
(D\Theta)(0,0)=\begin{pmatrix} I_n & A \\ -A & I_n\end{pmatrix}.
\ees
Thus, $\Theta$ is a smooth diffeomorphism from some neighborhood of the origin on to another. Let 
    \bes
        \Xi(x,y)=(u(x,y),v(x,y))
    \ees
denote the inverse of $\Theta$. Then,
    \be\label{E:inv}
        (D\Xi)(0,0)=
        \begin{pmatrix} D_xu & D_yu \\ D_xv & D_yv\end{pmatrix}(0,0)=\frac{1}{\prod_{j=1}^n(1+\lambda_j^2)}\begin{pmatrix} I_n & -A \\ A & I_n\end{pmatrix}.
    \ee
We claim the following properties of $\Theta$ (on some fixed neighborhood of $0\in\Cn$):
\begin{itemize}
    \item [(i)] $y(0,v)=v$,
    \item [(ii)] $\Theta\left(\{u=0\}\right)=\mathcal S$,
    \item [(iii)] there is a smooth $\rl^n$-valued map $\sigma$ on a neighborhood of $0\in\rl^n$ with $\sigma(0)=0$ and $D\sigma(0)=A$ such that 
    \bes
       \Theta\left(\{(u,v):v=\sigma(u)\}\right)= \mathcal R,
    \ees

    \item [(iv)] $\Theta_*\left(\nabla_u\Big|_{\{u=0\}}\right)=W\Big|_{\mathcal S}$, see the notation established in \eqref{E:PushVF},
    \item [(v)] $\Theta_*\left(\nabla_u\cdot b+\nabla_v\cdot c\right)=V$, where $b$ and $c$ are $(n\times n)$-matrix-valued maps given by
        \bea           b(u,v)&=&\left(D_xu\cdot (D\phi\circ \pi_x)-D_yu\right)(\Theta(u,v))
        \label{E:b}\\
        c(u,v)&=&\left(D_xv\cdot (D\phi\circ \pi_x)-D_yv\right)(\Theta(u,v)).
        \label{E:c}
        \eea
\end{itemize}
    Properties (i), (ii), (iv) and (v) follow from direct computation. Property (iii) follows from the implicit function theorem once we note that
    \beas
        \Theta^{-1}(\mathcal R)&=&\{(u,v):y(u,v)=\phi(x(u,v))\}\\
        &=&\{(u,v):F(u,v)=v-\left(A+(D\psi)(v)\right)\cdot u-\phi\left(u+A\cdot v+\psi(v)\right)=0\},
    \eeas
and, furthermore, $(D_uF)(0,0)=-A$ and $(D_vF)(0,0)=I_n$.

It now suffices to produce a smooth function $\wt r$ on a neighborhood of $0\in\Cn$ such that 
\bea
(\Theta^{-1}_*V)\wt r\Big|_{\Theta^{-1}(\mathcal R)}=(p\circ\pi_x\circ\Theta)\Big|_{\Theta^{-1}(\mathcal R)}\label{E:NewPDE1}\\
(\Theta^{-1}_*W)\wt r\Big|_{\Theta^{-1}(\mathcal S)}=(q\circ\pi_y\circ\Theta)\Big|_{\Theta^{-1}(\mathcal S)}\label{E:NewPDE2}
\eea
on some neighborhood of the origin, $(\nabla\wt r)(0)=0$, and $(\operatorname{Hess}\wt r)(0)=0$. Then, $r=\wt r\circ\Theta^{-1}$ is a smooth function on a neighborhood of the origin that satisfies \eqref{E:PDE1}
 and \eqref{E:PDE2}, and, owing to the chain rule \eqref{E:chain}, $(dd^cr)(0)=0$.  

 Note that \eqref{E:NewPDE1} and \eqref{E:NewPDE2}
are equivalent to
\bea
        (\nabla_u\wt r)(u,\sigma(u))\cdot B(u)+(\nabla_v\wt r)(u,\sigma(u))\cdot C(u)&=&P(u)\label{E:NPDE1}, \\
        (\nabla_u\wt r)(0,v)&=&Q(v),\label{E:NPDE2}
    \eea
respectively, 
where 
\beas    B(u)&=&b(u,\sigma(u)),\quad\quad\, \,\, 
    C(u)=c(u,\sigma(u)),\\
    P(u)&=&p(x(u,\sigma(u))),\quad 
    Q(v)=q(y(0,v))=q(v).
\eeas
We collect some relevant information. 
From \eqref{E:b}, \eqref{E:c}, \eqref{E:inv} and \eqref{E:phipsi}, we obtain that
    \bea
        B(0)&=&b(0,0)=-(D_yu)(0,0)=\frac{1}{\prod_{j=1}^n(1+\lambda_j^2)}A\label{E:B}\\
        C(0)&=&c(0,0)=-(D_yv)(0,0)=\frac{-1}{\prod_{j=1}^n(1+\lambda_j^2)}I_n.\label{E:C}
    \eea
From \eqref{E:pq}, it follows that 
\be
    P(0)=Q(0)=0\quad \text{and}\quad 
(DP)(0)=(DQ)(0)=0.\label{E:PQ}
\ee

We now produce an explicit $\wt r$ satisfying \eqref{E:NPDE1} and \eqref{E:NPDE2}. By \eqref{E:B}, \eqref{E:C}, and the fact that $(D\sigma)(0)=A$, we have that $\det\left(C(0)-(D\sigma)(0)\cdot B(0)\right)=-1$. Thus, there is some neighborhood $U$ of the origin in $\rl^n$ on which $C(u)-(D\sigma)(u)\cdot B(u)$ is defined and invertible for each $u$. Now, set
    \bes
        \wt r(u,v)= Q(v)\cdot u +\alpha(u)\cdot\left(v-\sigma(u)\right),\quad (u,v)\in U,
    \ees
where 
    \beas
        \alpha(u)&=&\left(P(u)-Q(\sigma(u))\cdot B(u)-u^T\cdot  (DQ)(\sigma(u))\cdot C(u)\right)\cdot \left(C(u)-(D\sigma)(u)\cdot B(u)\right)^{-1}\\
        &=:&\beta(u)\cdot\gamma(u).
    \eeas
From \eqref{E:PQ}, we have that  
    \be\label{E:alpha0}
    \alpha(0)=\beta(0)=0.
    \ee
Using the product rule \eqref{E:prod}, \eqref{E:PQ}, and that $\beta(0)=0$, we have that
    \beas
        (D\alpha)(0)&=&\gamma^T(0)\cdot (D\beta)(0)+(I_n\otimes \beta(0))\cdot (D\gamma)(0)\\        &=&\gamma^T(0)\cdot\left((DP)(0)+B(0)^T\cdot (DQ)(0)\cdot(D\sigma)(0)+(I_n\otimes Q(0))\cdot (DB)(0)\right)\\
        &&\qquad +\gamma^T(0)\cdot\left(C(0)^T\cdot (DQ)^T(0)\right) \\
        &=&0.
    \eeas

We now check that $\wt r$ satisfies all the desired properties. By the product rule \eqref{E:prod}, 
\beas
    (\nabla_u\wt r)(u,v)&=&Q(v)+\left(v-\sigma(u)\right)^T\cdot (D\alpha)(u)- \alpha(u)\cdot (D\sigma)(u)\\
    (\nabla_v\wt r)(u,v)&=& u^T\cdot (DQ)(v)+\alpha(u)\\
    (\operatorname{Hess}_{uu}\wt r)(u,v)&=& -(D\alpha)^T(u)\cdot (D\sigma)^T(u)+ \left(I_n\otimes (v-\sigma(u))^T\right)\cdot (D^2\alpha)(u)\\ 
    &&\qquad -(D\sigma)^T(u)\cdot (D\alpha)(u)-\left(I_n\otimes \alpha(u)\right) \cdot (D^2\sigma)(u)\\
    (\operatorname{Hess}_{vv}\wt r)(u,v)&=& \left(I_n\otimes u^T\right)\cdot (D^2Q)(v)\\
    (\operatorname{Hess}_{uv}\wt r)(u,v)&=&(DQ)(v)+(D\alpha)^T(u).
\eeas
Substituting $\nabla_u\wt r$ and $\nabla_v\wt r$ along $\{v=\sigma(u)\}$ in the left-hand side of \eqref{E:NewPDE1}, we obtain
\beas
&&Q(\sigma(u))\cdot B(u)-\alpha(u)\cdot (D\sigma)(u)\cdot B(u)+u^T\cdot (DQ)(\sigma(u))\cdot C(u)+\alpha(u)\cdot C(u)\\
&=& \alpha(u)\cdot\left(C(u)-(D\sigma)(u)\cdot B(u)\right)+Q(\sigma(u))\cdot B(u)+u^T\cdot (DQ)(\sigma(u))\cdot C(u)\\
&=& P(u).
\eeas
Thus, \eqref{E:NPDE1} holds. Next, from \eqref{E:alpha0}, we have that
\bes
(\nabla_v\wt r)(0,v)=\alpha(0)=0.
\ees
Thus, \eqref{E:NPDE2} holds. Finally, since $Q(0)=\sigma(0)^T=\alpha(0)=0$ and $(D\alpha)(0)=(DQ)(0)=0$,
\bes
    \nabla\wt r(0)=0\quad \text{and}\quad \operatorname{Hess}\wt r(0)=0.
\ees
We have, thus, produced $\wt r$ as desired, and the proof is complete.
\end{proof}

\section{Proof of Theorem~\ref{T:main}}\label{S:proof}

Let $\iota:M\rightarrow\Cn$ be as given. Suppose $\iota(M)$ is Lagrangian with respect to some K{\"a}hler form $\omega$ on $\Cn$. Then, by \cite{Ga00}, $\iota(M)$ is rationally convex in $\Cn$. Moreover, at each double point $p\in\iota(M)$, the union of the pair of transverse totally real tangent planes to $\iota(M)$ at $p$ is Lagrangian with respect to the linear K{\"a}hler form whose coefficients coincide with the coefficients of $\omega$ at $p$. Thus, by Proposition~\ref{P:planes}, $p$ is a diagonalizable double point of $\iota(M)$. 

Conversely, suppose $\iota(M)$ is rationally convex in $\Cn$ and the double points $p_1,...,p_m$ of $\iota(M)$ are all diagonalizable. We denote by $\mathbb B^n(p,r)$ the open Euclidean ball in $\Cn$ of radius $r>0$ and center $p\in\Cn$, and by $\overline {\mathbb B^n}(p,r)$, its closure. By Lemma~\ref{L:local}, there is an $\eps>0$ and smooth functions $f_j$ on $\mathbb B^n(p_j,\eps)$, $j=1,...,m$, so that for each $j=1,...,m$. 
\begin{itemize}
    \item [(i)] $\mathbb B^n (p_j,\eps)\cap \mathbb B^n (p_k,\eps)=\emptyset$ for all $k\neq j$,
    \item [(ii)] $dd^c\!f_j>0$ on $\mathbb B^n (p_j,\eps)$, and
    \item [(iii)] $\iota^*(d^c\!f_j)=0$ on $M\cap \iota^{-1}(\mathbb B^n(p_j,\eps))$.
\end{itemize}
Let $\de\in(0,\eps)$ be such that there is a smooth decreasing cut-off function $\chi:[0,\infty)\rightarrow[0,1]$ satisfying $\chi\equiv 1$ on $[0,\de]$ and $\chi\equiv 0$ on $[\eps,\infty)$. 
By \cite[Proposition~3.1]{Mi20}, there is a smooth function $\varphi:\Cn\rightarrow \rl$ such that 
    \begin{itemize}
        \item [(a)] $\varphi$ is plurisubharmonic on $\Cn$,
        \item [(b)] $\varphi$ is strictly plurisubharmonic on $\Cn\setminus \bigcup_{j=1}^m\mathbb B^n(p_j,\de/2)$,
        \item [(c)] $\iota^*dd^c\varphi=0$ on $M$.
    \end{itemize}
Since $d\chi\equiv 0$ on $\mathbb B^n(0,\de)$ and each $\overline{\mathbb B^n}(p_j,\eps)\setminus\mathbb B^n(p_j,\de)$ is a compact region in the set where $\varphi$ is strictly plurisubharmonic, there is a $C>0$ such that 
    \be\label{E:pos}
Cdd^c\varphi+\sum_{j=1}^md\chi_j\wedge d^c\!f_j>0,\quad \text{on }\Cn\setminus\bigcup_{j=1}^m\mathbb B^n(p_j,\de/2),
    \ee
where $\chi_j=\chi\left(\Vert z-p_j\Vert\right)$ on $\Cn$. Now, set 
        \bes        \omega=Cdd^c\varphi+\sum_{j=1}^md\left(\chi_jd^c\!f_j\right)
    \ees
on $\Cn$, where each $\chi_j d^c\!f_j$ is extended to all of $\Cn$ by setting it as zero outside $\mathbb B^n(p_j,\eps)$. 

Clearly $\om$ is a closed $(1,1)$-form. Moreover, 
\bes    \omega=Cdd^c\varphi+\sum_{j=1}^md\chi_j\wedge d^c\!f_j+\sum_{j=1}^m\chi_jdd^c\!f_j
\ees
is positive on $\cup_{j=1}^m\mathbb B^n(p_j,\de)$ owing to (ii), (a) and that $d\chi_j\equiv 0$ on $\mathbb B^n(p_j,\de)$. It is positive on $\Cn\setminus \cup_{j=1}^m\mathbb B^n(p_j,\de)$ owing to \eqref{E:pos} and the nonnegativity of each $\chi_jdd^cf_j$ on $\Cn$. Finally,
\bes    \iota^*\omega=C\iota^*dd^c\varphi+\sum_{j=1}^m\iota^*d\chi_j\wedge \iota^*d^c\!f_j+\sum_{j=1}^m\chi_jd(\iota^*d^c\!f_j)=0
\ees
owing to (c), (iii), and the fact that each $\chi_j\equiv 0$ outside $\mathbb B^n(p_j,\eps)$. Thus, $\iota(M)$ is Lagrangian with respect to the global K{\"a}hler form $\om$. This completes the proof of Theorem~\ref{T:main}.

\bibliography{RatCvx}{}

\begin{thebibliography}{1}

\bibitem{Du91}
J.~Duval.
\newblock Convexit{\'e} rationnelle des surfaces lagrangiennes.
\newblock {\em Invent. Math.}, 104(1):581--599, 1991.

\bibitem{Du94}
J.~Duval.
\newblock Une caract{\'e}risation k{\"a}hl{\'e}rienne des surfaces
  rationnellement convexes.
\newblock {\em Acta Math.}, 172:77--89, 1994.

\bibitem{DuGa08}
J.~Duval and D.~Gayet.
\newblock Rational convexity of non-generic immersed {L}agrangian submanifolds.
\newblock {\em Math. Ann.}, 345(1):25--29, 2008.

\bibitem{DuSi95}
J.~Duval and N.~Sibony.
\newblock Polynomial convexity, rational convexity, and currents.
\newblock {\em Duke Math J.}, 97(2):487--513, 1995.

\bibitem{Ga00}
D.~Gayet.
\newblock Convexit{\'e} rationnelle des sous-vari{\'e}t{\'e}s immerg{\'e}es
  lagrangiennes.
\newblock {\em Ann. Scient. {\'E}c. Norm. Sup.}, 33(2):291--300, 2000.

\bibitem{Mi20}
O.~Mitrea.
\newblock A characterization of rationally convex immersions.
\newblock {\em J. Geom. Anal.}, 30(1):968--986, 2020.

\bibitem{We88}
B.~M. Weinstock.
\newblock On the polynomial convexity of the union of two maximal totally real
  subspaces of {$\Cn$}.
\newblock {\em Math. Ann.}, 282(1):131--138, 1988.

\end{thebibliography}
\bibliographystyle{plain}
\end{document}